\documentclass[11pt]{amsart}
\usepackage{amssymb}
\def\Li{\hbox{  Li}}

\newtheorem{lemme}{Lemma}
\newtheorem{coro}{Corollary}
\newtheorem{theorem}{Theorem}

\begin{document}

\title [logarithmically completely monotonic functions]{ A Class of logarithmically completely monotonic functions relating the $q$-gamma function and applications\\}%

\author[K. Mehrez]{ Khaled Mehrez }
 \address{Khaled Mehrez. D\'epartement de Math\'ematiques ISSAT Kasserine, Tunisia.}
 \email{k.mehrez@yahoo.fr}
\begin{abstract}
In this paper, the logarithmically complete monotonicity property for a functions involving $q$-gamma function is investigated for $q\in(0,1).$  As applications of this results, some new inequalities for the $q$-gamma function  are established. Furthermore, let the sequence $r_n$ be defined by
$$n!=\sqrt{2\pi n}(n/e)^n e^{r_n}.$$
We establish new estimates for Stirling's formula remainder $r_n.$ 
\end{abstract}
\maketitle


\section{\textbf{Introduction}}

A real valued function $f$, defined on an interval $I,$ is called completely monotonic, if f has derivatives of all orders and satisfies
\begin{equation}
(-1)^n f^{(n)}(x)\geq 0,\;\;\;n\in\mathbb{N}_0,\;x\in I,
\end{equation}
where  $\mathbb{N}$ the set of all positive integers.

A positive function $f$ is said to be logarithmically completely monotonic on an interval $I$ if its logarithm $\log f$ satisfies
$$(-1)^n \Big(\log f(x)\Big)^{(n)}(x)\geq 0,$$
for all $x\in I$ and $n\in\mathbb{N}.$

Completely monotonic functions have remarkable applications in different
branches of mathematics. For instance, they play a role in potential theory, probability theory, physics, numerical and asymptotic analysis, and combinatorics (see \cite{C1} and the references given therein).

The $q$-analogue of the gamma function is defined as
\begin{equation}
\Gamma_q(x)=(1-q)^{1-x}\prod_{j=0}^{\infty}\frac{1-q^{j+1}}{1-q^{j+x}},\:0<q<1,
\end{equation}
and 
\begin{equation}
\Gamma_q(x)=(q-1)^{1-x}q^{\frac{x(x-1)}{2}}\prod_{j=0}^{\infty}\frac{1-q^{-(j+1)}}{1-q^{-(j+x)}},\:q>1.
\end{equation}
The  $q-$gamma  function  $\Gamma_q(z)$  has  the  following  basic 
properties:
\begin{equation}
\lim_{q\longrightarrow1^{-}}\Gamma_{q}(z)=\lim_{q\longrightarrow1^{+}}\Gamma_{q}(z)=\Gamma(z),
\end{equation}
and 
\begin{equation}\label{369}
\Gamma_{q}(z)=q^{\frac{(x-1)(x-2)}{2}}\Gamma_{\frac{1}{q}}(z).
\end{equation}

The $q-$digamma function $\psi_q,$ the  $q-$analogue  of  the  psi  or  digamma  function $\psi$ is 
defined for $0<q<1$ by 
\begin{equation}\label{ttt}
\begin{split}
\psi_q(x)&=\frac{\Gamma^{'}_q(x)}{\Gamma_q(z)}\\
&=-\log(1-q)+\log q \sum_{k=0}^{\infty}\frac{q^{k+x}}{1-q^{k+x}}\\
&=-\log(1-q)+\log q \sum_{k=1}^{\infty}\frac{q^{kx}}{1-q^{k}}.
\end{split}
\end{equation}
For $q>1$ and $x>0$, the $q-$digamma function $\psi_q$ is defined by 
\begin{equation*}
\begin{split}
\psi_q(x)&=-\log(q-1)+\log q\left[x-\frac{1}{2}-\sum_{k=0}^{\infty}\frac{q^{-(k+x)}}{1-q^{-(k+x)}}\right]\\
&=-\log(q-1)+\log q\left[x-\frac{1}{2}-\sum_{k=1}^{\infty}\frac{q^{-kx}}{1-q^{-kx}}\right]
\end{split}
\end{equation*}
From the previous definitions, for a positive $x$ and $q>1$, we get 
$$\psi_q(x)=\frac{2x-3}{2}\log q+\psi_{1/q}(x).$$

Using the Euler-Maclaurin formula, Moak \cite{M} obtained the following q-analogue of Stirling formula 
\begin{equation}
\log \Gamma_q(x)\sim \left(x-\frac{1}{2}\right)\log\left(\frac{1-q^x}{1-q}\right)+\frac{\Li_2(1-q^x)}{\log q}+\frac{1}{2}H(q-1)\log q+C_{\hat{q}}+\sum_{k=1}^{\infty}\frac{B_{2k}}{(2k)!}\left(\frac{\log \hat{q}}{\hat{q}^{x} -1}\right)^{2k-1} \hat{q}^{x} P_{2k-3}(\hat{q}^{x}) 
\end{equation}
as $x\longrightarrow\infty$ where $H(.)$ denotes the Heaviside step function, $B_k,\;k=1,2,...$ are the Bernoulli numbers,
$$\hat{q}=\begin{cases}
q & if\,\,0<q<1\\
1/q & if\: q>1\end{cases}$$
$\Li_2(z)$ is the dilogarithm function defined for complex argument $z$ as \cite{abra}
\begin{equation}
 \Li_{2}(z)=-\int_{0}^{z}\frac{\log(1-t)}{t}dt,\;z\notin (0,\infty)
 \end{equation} 
$P_k$ is a polynomial of degree $k$ satisfying
\begin{equation}
P_k(z)=(z-z^2)P^{'}_{k-1}(z)+(kz+1)P_{k-1}(z),\;P_0=P_{-1}=1,\; k=1,2,... 
\end{equation}
and 
$$C_k=\frac{1}{2}\log (2\pi)+\frac{1}{2}\log\left(\frac{q-1}{\log q}\right)-\frac{1}{24}\log q+\log\left(\sum_{m=-\infty}^{\infty}r^{m(6m+1)}-r^{(2m+1)(3m+1)}\right),$$
where $r=\exp(4\pi^2/\log q).$ It is easy to see that
$$\lim_{q\longrightarrow 1}C_q=\frac{1}{2}\log(2\pi),\;\;\textrm{and}\;\;\lim_{q\longrightarrow 1}\frac{\Li_2(1-q^x)}{\log q}=-x.$$

Stirling's formula
\begin{equation}
n!\sim\sqrt{2\pi n}\left(\frac{n}{e}\right)^n, \;n\in\mathbb{N}
\end{equation}
has many applications in statistical physics, probability theory and number theory. Actually, it was first discovered in 1733
by the French mathematician Abraham de Moivre (1667-1754) in the form
$$n!\sim constant \sqrt{n}\left(\frac{n}{e}\right)^n$$
when he was studying the Gaussian distribution and the central limit theorem. Afterwards, the Scottish mathematician
James Stirling (1692-1770) found the missing constant $\sqrt{2\pi}$ when he was trying to give the normal approximation of the
binomial distribution.

In 1940 Hummel \cite{H} defined the sequence $r_n$ by
\begin{equation}
n!=\sqrt{2\pi n}(n/e)^n e^{r_n},
\end{equation}
and established
\begin{equation}\label{333}
\frac{11}{2}<r_n+\log \sqrt{2\pi}<1
\end{equation}
After the inequality (\ref{333}) was published, many improvements have been given. For example, Robbins \cite{R} established
\begin{equation}
\frac{1}{12n+1}<r_n<\frac{1}{12n}
\end{equation}

The main aim of this paper is to investigate the logarithmic complete monotonicity property
of the function
\begin{equation}\label{001}
f_{\alpha,\beta}(x;q)=\frac{\Gamma_{q}(x+\beta)\exp\left(\frac{-\Li_{2}(1-q^{x})}{\log q}\right)}{\left(\frac{1-q^x}{1-q}\right)^{x+\beta-\alpha}},\;x>0, \end{equation}
for all reals $\alpha, \beta$ and $q$ such that $q\in(0,1).$ As applications of these results, sharp bounds for the $q$-gamma function are derived. In addition, we present new estimate for for Stirling's formula remainder $r_n,$ (see Corollary \ref{c4}). Some results are shown to be a generalization of results which were obtained by Chen and Qi \cite{qi}.

\section{\textbf{logarithmically completely monotonic function related the $q$-gamma function}}
In order to study the function defined by (\ref{001}) we need the following lemma which
is considered the main tool to arrive at our results.
\begin{lemme}\label{salem2}\cite{salem}
For every $x,q\in\mathbb{R}_+$, there exists at least one real number $a \in [0, 1]$ such that 
\begin{equation}
\psi_q(x)=\log\left(\frac{1-q^{x+a}}{1-q}\right)+\frac{q^{x}\,\log q}{1-q^{x}}-\left(\frac{1}{2}-a\right)H(q-1)\log q
\end{equation}
where $H(.)$ is the Heaviside step function.
\end{lemme}
\begin{theorem}\label{t1} Let $\alpha$ be a real number. The function $f_{\alpha,1}(q;x)$ is logarithmically completely monotonic on $(0,\infty)$, if and only if $2\alpha\leq1.$
\end{theorem}
\begin{proof}
Taking logarithm of $f_{\alpha, 1} (x;q)$ leads to
\begin{equation}
\log f_{\alpha,1} (x;q)=\log\Gamma_{q}(x+1)-(x+1-\alpha)\log\left(\frac{1-q^{x}}{1-q}\right)-\frac{\Li_{2}(1-q^x)}{\log q}.
\end{equation}
Differentiation yields
\begin{equation}
\left(\log f_{\alpha,1}(x;q)\right)^{'}=\psi_{q}(x+1)-\log\left(\frac{1-q^{x}}{1-q}\right)+\left(1-\alpha\right)\frac{\log q\;q^{x}}{1-q^{x}}.
\end{equation}
From the series expansion 
$$\frac{1}{(1-x)^2}=\sum_{k=0}^{\infty}(k+1)x^k,$$
for $x\in(0,1)$ and (\ref{ttt}) we get
\begin{equation}
\begin{split}
\left(\log f_{\alpha,1}(x;q)\right)^{''}&=\psi_{q}^{'}(x+1)+(1-\alpha)\frac{q^x}{(1-q^x)^2}+\log q\frac{q^x}{1-q^x}\\
&=\big(\log q\big)^2\sum_{k=1}^{\infty}\frac{kq^{k(x+1)}}{1-q^k}+(1-\alpha)\big(\log q\big)^2\sum_{k=1}^{\infty}kq^{kx}+\log q\sum_{k=1}^{\infty} q^{kx}\\
&=\sum_{k=1}^{\infty}\frac{\log q q^{kx}}{1-q^k}\Phi_{\alpha,1}(q^k),
\end{split}
\end{equation}
where
$$\Phi_{\alpha,1}(y)=y\log y+(1-\alpha)(1-y)\log y+(1-y),\; y=q^k,\;k=1,2,...$$
In order to determine the sign of the function $\Phi_{\alpha,1}(y)$, we have
\begin{equation}\label{kh1}
\begin{split}
\Phi_{\alpha,1}(y)&=y\big(-\log(1/y)+(1-\alpha)\log(1/y)(1-1/y)+1/y-1 \big)\\
&=y\sum_{k=2}^{\infty}\frac{(\log(1/y))^k}{(k-1)!}\big[\alpha-1+\frac{1}{k}\big].
\end{split}
\end{equation}
Therefore, the function $\Phi_{\alpha,1}(y)$ is less than zero if $2\alpha\leq1.$ Thus implies that the function $\left(\log f_{\alpha,1}(x;q)\right)^{''}$ is completely monotonic on $(0,\infty).$ This can be rewritten as 
$$(-1)^n\left(\log f_{\alpha,1}(x;q)\right)^{(n)}\geq0,\;n\geq2.$$
In particular, $\left(\log f_{\alpha,1}(x;q)\right)^{''}\geq0,$ so $\left(\log f_{\alpha,1}(x;q)\right)^{'}$ is increasing on $(0,\infty),$ and consequently
\begin{equation*}
\begin{split}
\left(\log f_{\alpha,1}(x;q)\right)^{(1)}&\leq \lim_{x\longrightarrow\infty}\left(\log f_{\alpha,1}(x;q)\right)^{(1)}\\
&=\lim_{x\longrightarrow\infty}\Bigg(\psi_{q}(x+1)-\log\left(\frac{1-q^{x}}{1-q}\right)+\left(1-\alpha\right)\frac{\log q\; q^{x}}{1-q^{x}}\Bigg)\\
&=0.
\end{split}
\end{equation*} 
So $f_{\alpha,1}$ is logarithmically completely monotonic on $(0,\infty)$ if $2\alpha\leq1.$

Conversely, If the function $f_{\alpha,1}(x;q)$ is logarithmically completely monotonic on $(0,\infty)$, then for all real $x>0$, 
 \begin{equation}\label{mp}
 \left(\log f_{\alpha,1}(x;q)\right)^{'}=\psi_q(x+1)-\log\left(\frac{1-q^{x}}{1-q}\right)+\left(1-\alpha\right)\frac{\log q\;q^{x}}{1-q^{x}}\leq 0.
 \end{equation}
From the equation (\ref{mp}) and along with the identity
\begin{equation}
\psi_q(x+1)=\psi_q(x)-\frac{q^x \log q}{1-q^x},
\end{equation}
we have  
$$ \left(\log f_{\alpha,1}(x;q)\right)^{'}=\psi_q(x)-\log\left(\frac{1-q^{x}}{1-q}\right)-\alpha\log q\frac{q^{x}}{1-q^{x}}\leq0,$$
 which is equivalent to
 \begin{equation}\label{rr1}
 \psi_q(x)-\log\left(\frac{1-q^{x}}{1-q}\right)\leq\alpha\frac{\log q\;q^{x}}{1-q^{x}}.
   \end{equation}
   It is worth mentioning that, Moak \cite{M} proved the following approximation for the $q$-digamma function
   $$\psi_q(x)=\log\left(\frac{1-q^x}{1-q}\right)+\frac{1}{2}\frac{\log q\;q^{x}}{1-q^{x}}+O\left(\frac{\log^2 q\;q^{x}}{(1-q^{x})^2}\right)$$
   holds for all $q>0$ and $x>0$ and so $\psi_q(x)\sim I(x;q)$ on $(0,\infty)$ where 
   \begin{equation}\label{0rr2}
   I(x;q)=\log\left(\frac{1-q^x}{1-q}\right)+\frac{1}{2}\frac{\log q\;q^{x}}{1-q^{x}}
   \end{equation}
  Combining (\ref{rr1}) and (\ref{0rr2}) we have
   $$\alpha\leq \frac{1}{2}.$$
   The proof is complete.
   
\end{proof}
\begin{theorem} \label{t2}Let $\alpha$ be a real number. The function $[f_{\alpha,1}(q;x)]^{-1}$ is logarithmically completely monotonic on $(0,\infty)$, if and only if $\alpha\geq1.$
\end{theorem}
\begin{proof} From (\ref{kh1}), we conclude that the function $\Phi_{\alpha,1}(y)\geq0$ if $\alpha\geq1$ we conclude that 
   $$(-1)^n\Bigg(\log\frac{1}{f_{\alpha,1}(q;x)}\Bigg)^{(n)}\geq0$$
   for all $x>0,\;\alpha\geq1,\;q\in(0,1).$ and $n\geq2.$ So, 
  $$\Bigg(\log\frac{1}{f_{\alpha,1}(q;x)}\Bigg)^{(1)}=\log\left(\frac{1-q^{x}}{1-q}\right)-\psi_q(x)+\alpha\log q\frac{q^{x}}{1-q^{x}},$$
  is increasing, thus 
  \begin{equation*}
  \begin{split}
   \Bigg(\log\frac{1}{f_{\alpha,1}(q;x)}\Bigg)^{(1)} &<\lim_{x\longrightarrow\infty}\Bigg(\log\frac{1}{f_{\alpha,1}(q;x)}\Bigg)^{(1)}\\
    &=\lim_{x\longrightarrow\infty} \Bigg(\log\left(\frac{1-q^{x}}{1-q}\right)-\psi_q(x)+\alpha\log q\frac{q^{x}}{1-q^{x}}\Bigg)\\
    &=0.
    \end{split}
    \end{equation*}
    Hence, For $\alpha\geq1$ and $n\in\mathbb{N},$
    \begin{equation*}
    (-1)^n \Bigg(\log\frac{1}{f_{\alpha,1}(q;x)}\Bigg)^{(n)}\geq0,
    \end{equation*} 
    on $(0,\infty).$
Now, assume that $\frac{1}{f_{\alpha,1}(q;x)}$ is logarithmically completely monotonic on $(0,\infty)$, by definition, this give us that for all $q\in(0,1)$    and $x>0,$
  $$\Bigg(\log\frac{1}{f_{\alpha,1}(q;x)}\Bigg)^{(1)}=\log\left(\frac{1-q^{x}}{1-q}\right)-\psi_q(x)+\alpha\log q\frac{q^{x}}{1-q^{x}}\leq0,$$
 which implies that
 \begin{equation}\label{ll}
 \alpha \geq \frac{1-q^x}{\log q\;q^x}\Bigg(\psi_q(x)-\log\Bigg(\frac{1-q^x}{1-q}\Bigg)\Bigg).
  \end{equation}
In  view of Lemma \ref{salem2} and inequality (\ref{ll}), we see that for all $x>0$ and $q\in(0,1)$ there exists at least one real number $a\in[0,1]$ such that 
\begin{equation*}
\alpha \geq \frac{1-q^x}{\log q\;q^x}\Bigg(\log\Bigg(\frac{1-q^{x+a}}{1-q^x}\Bigg)+\frac{\log q\;q^x}{1-q^x}\Bigg),
\end{equation*}
and consequently
$$\alpha\geq1$$
as $x\longrightarrow\infty.$ This ends the proof.
\end{proof}
\begin{theorem} \label{t3} Let $\alpha$ be a real number and $\beta\geq0.$ Then, the function $f_{\alpha,\beta}(x;q)$ is logarithmically completely monotonic function on $(0,\infty)$ if $2\alpha\leq1\leq\beta.$ 
 \end{theorem}
 \begin{proof} 
 Standard calculations lead us to
 \begin{equation*}
\begin{split}
\left(\log f_{\alpha,\beta}(x;q)\right)^{''}&=\psi_{q}^{'}(x+\beta)+(\beta-\alpha)\frac{q^x}{(1-q^x)^2}+\log q\frac{q^x}{1-q^x}\\
&=\big(\log q\big)^2\sum_{k=1}^{\infty}\frac{kq^{k(x+\beta)}}{1-q^k}+(\beta-\alpha)\big(\log q\big)^2\sum_{k=1}^{\infty}kq^{kx}+\log q\sum_{k=1}^{\infty} q^{kx}\\
&=\sum_{k=1}^{\infty}\frac{\log q.q^{kx}}{1-q^k}\Phi_{\alpha,\beta}(q^k),
\end{split}
\end{equation*}
where
$$\Phi_{\alpha,\beta}(y)=y^\beta\log y+(\beta-\alpha)(1-y)\log y+(1-y),\; y=q^k,\;k=1,2,...$$
Thus
$$\Phi_{\alpha,\beta}(y)=y^\beta\left(\sum_{k=2}^{\infty}\frac{\left(\log(1/y)\right)^k}{(k-1)!}\left[\frac{\beta^k-(\beta-1)^k}{k}+(\beta-\alpha)[(\beta-1)^{k-1}-\beta^{k-1}]\right]\right).$$
From the inequality \cite{qi}
$$\beta^k-(\beta-1)^k<k(\beta-\alpha)(\beta^{k-1}-(\beta-1)^{k-1})$$
we conclude that $\Phi_{\alpha,\beta}(y)\leq0.$ This implies that for $n\geq2$ 
\begin{equation}\label{xx}
(-1)^{n}\left(\log f_{\alpha,\beta}(x;q)\right)^{(n)}\geq0
\end{equation}
on $(0,\infty)$ for $2\alpha\leq1\leq\beta.$
As $\left(\log f_{\alpha,\beta}(x;q)\right)^{(2)}\geq0$, it follows that $\left(\log f_{\alpha,\beta}(x;q)\right)^{(1)}$ is increasing on $(0,\infty)$, and consequently
\begin{equation*}
\begin{split}
\left(\log f_{\alpha,\beta}(x;q)\right)^{(1)}&\leq \lim_{x\longrightarrow\infty}\left(\log f_{\alpha,\beta}(x;q)\right)^{(1)}\\
&=\lim_{x\longrightarrow\infty}\Bigg(\psi_{q}(x+\beta)-\log\left(\frac{1-q^{x}}{1-q}\right)+\left(\beta-\alpha\right)\frac{\log q\; q^{x}}{1-q^{x}}\Bigg)\\
&=0.
\end{split}
\end{equation*}
In conclusion, (\ref{xx}) is true also $n=1$, and we conclude that the function $f_{\alpha,\beta}(x;q)$ is logarithmically completely monotonic on $(0,\infty)$ for $2\alpha\leq 1\leq \beta.$  The proof is now completed.  
 \end{proof}
\section{\textbf{Inequalities}}
 As applications of the logarithmic complete monotonicity properties of the function (\ref{001})
which are proved in Theorem \ref{t1}, Theorem \ref{t2} and Theorem \ref{t3}, we can provide the following inequalities for the
$q$-gamma functions.
 
 \begin{coro} Let $q\in(0,1), \;n\in\mathbb{N}$ and $x_k>0\;\;(1\leq k\leq n).$ Suppose that $$\sum_{k=1}^n p_k=1\;\;(p_k\geq0).$$
  If $2\alpha\leq1\leq\beta$, then 
 \begin{equation}\label{111}
 \begin{split}
 \frac{\Gamma_{q}\Big(\sum_{k=1}^{n}p_{k}x_{k}+\beta\Big)}{\prod_{k=1}^{n}\Big[\Gamma_{q}(x_{k}+\beta)\Big]^{p_{k}}}&\leq \frac{\left(\frac{1-q^{\sum_{k=1}^{n}p_{k}x_{k}}}{1-q}\right)^{\sum_{k=1}^{n}p_{k}x_{k}+\beta-\alpha}}{\prod_{k=1}^{n}\left(\frac{1-q^{x_{k}}}{1-q}\right)^{p_{k}(x_{k}+\beta-\alpha)}}\exp\left(\frac{Li_{2}\left(1-q^{\sum_{k=1}^{n}p_{k}x_{k}}\right)-\sum_{k=1}^{n}p_{k}Li_{2}(1-q^{x_{k}})}{\log q}\right)
 \end{split}
 \end{equation}
 \end{coro}
 \begin{proof}
 From Theorem \ref{t3}, $f_{\alpha,\beta}(x;q)$ is logarithmically completely monotonic on the interval $(0,\infty),$ which also implies that the function $f_{\alpha,\beta}(x;q)$ is logarithmically convex. Combining this fact with Jensen's inequality for convex functions yields
 \begin{equation}\label{pp}
 \log f_{\alpha,\beta}\left(\sum_{k=1}^{n}p_kx_k;q\right)\leq \sum_{k=1}^{n} p_k\log f_{\alpha,\beta}(x_k;q).
 \end{equation}
 Rearranging (\ref{pp}) can lead to the inequality (\ref{111}).
 \end{proof}
 \begin{coro}
 Let $q\in(0,1), \;n\in\mathbb{N}$ and $x_k>0\;\;(1\leq k\leq n).$ Suppose that $$\sum_{k=1}^n p_k=1\;\;(p_k\geq0).$$
  Then, the following inequalities holds
 \begin{equation}\label{222}
 \begin{split} \frac{\left(\frac{1-q^{\sum_{k=1}^{n}p_{k}x_{k}}}{1-q}\right)^{\sum_{k=1}^{n}p_{k}x_{k}}}{\prod_{k=1}^{n}\left(\frac{1-q^{x_{k}}}{1-q}\right)^{p_{k}x_{k}}}\exp\left(\frac{\Li_{2}\left(1-q^{\sum_{k=1}^{n}p_{k}x_{k}}\right)-\sum_{k=1}^{n}p_{k}\Li_{2}(1-q^{x_{k}})}{\log q}\right)
 \end{split}
 \end{equation}
  $$\leq  \frac{\Gamma_{q}\Big(\sum_{k=1}^{n}p_{k}x_{k}+1\Big)}{\prod_{k=1}^{n}\Big[\Gamma_{q}(x_{k}+1)\Big]^{p_{k}}}$$
  $$\leq \frac{\left(\frac{1-q^{\sum_{k=1}^{n}p_{k}x_{k}}}{1-q}\right)^{\sum_{k=1}^{n}p_{k}x_{k}+1/2}}{\prod_{k=1}^{n}\left(\frac{1-q^{x_{k}}}{1-q}\right)^{p_{k}(x_{k}+1/2)}}\exp\left(\frac{\Li_{2}\left(1-q^{\sum_{k=1}^{n}p_{k}x_{k}}\right)-\sum_{k=1}^{n}p_{k}\Li_{2}(1-q^{x_{k}})}{\log q}\right)$$
 \end{coro}
 \begin{proof} The right side inequality of (\ref{222}) follows by inequality (\ref{111}). From Theorem \ref{t2}, the function $f_{1,1}(x;q)$ is logarithmically concave. Combining this fact with Jensen's inequality for convex functions we obtain the left side inequality of (\ref{222}).
 \end{proof}
\begin{coro} Let $q\in(0,1)$ and $a,b$ be a reals numbers such that $0<a<b$. Then the following inequalities
\begin{equation}\label{sss}
\frac{\left(\frac{1-q^{b}}{1-q}\right)^{b-1}}{\left(\frac{1-q^{a}}{1-q}\right)^{a-1}}\exp\left(\frac{\Li_{2}(1-q^{b})-Li_{2}(1-q^{a})}{\log q}\right)
\leq\frac{\Gamma_{q}(b)}{\Gamma_{q}(a)}\leq\frac{\left(\frac{1-q^{b}}{1-q}\right)^{b-\frac{1}{2}}}{\left(\frac{1-q^{a}}{1-q}\right)^{a-\frac{1}{2}}}\exp\left(\frac{\Li_{2}(1-q^{b})-\Li_{2}(1-q^{a})}{\log q}\right)
 \end{equation}
 holds.
 \end{coro}
 \begin{proof} From the monotonicity of the functions $f_{\frac{1}{2},1}(q;x)$ and $[f_{1,1}(q;x)]^{-1}$ and the recurrence formula
 \begin{equation}
 \Gamma_{q}(x+1)=\frac{1-q^x}{1-q}\Gamma_q(x),
 \end{equation}
 we obtain the inequalities (\ref{sss}).
 \end{proof}
 \begin{coro}\label{c4} Let $q\in(0,1)$ the following inequalities
 \begin{equation}\label{kh22}
 \left(\frac{1-q^{x}}{1-q}\right)^{x}\exp\left(-\frac{\Li_{2}(1-q^{})}{\log q}\right)\exp\left(\frac{\Li_{2}(1-q^{x})}{\log q}\right)
 \leq\Gamma_q(x+1)
 \end{equation}
 $$\leq \left(\frac{1-q^{x}}{1-q}\right)^{x+1/2}\exp\left(-\frac{\Li_{2}(1-q^{})}{\log q}\right)\exp\left(\frac{\Li_{2}(1-q^{x})}{\log q}\right)$$
 holds for all $x\in[1,\infty).$
 \end{coro}
 \begin{proof} As the function $f_{1/2,1}(q;x)$ is logarithmically completely monotonic, $f_{1/2,1}(q;x)$ is also decreasing. The following inequality hold true for every $x\geq1:$
 \begin{equation}\label{kh97}
f_{1/2,1}(q;x)\leq f_{1/2,1}(q;1).
\end{equation}
In addition, as $1/f_{1,1}(q;x)$ is logarithmically completely monotonic, we deduce that $f_{1,1}(q;x)$ is increasing. The following inequality hold true for all $x\geq1:$
\begin{equation}\label{kh98}
 f_{1,1}(q;1)\leq f_{1,1}(q;x).
 \end{equation}
Combining inequalities(\ref{kh97}) and (\ref{kh98}) we obtain the inequalities (\ref{kh22}).   
 \end{proof}
 
 In the next Corollary we present new estimates for Stirling's formula remainder $r_n$.
 
 \begin{coro}The following inequalities hold true for every integer $n\geq1:$
 \begin{equation}\label{rr}
 e.\left(\frac{n}{e}\right)^n\leq n!\leq e.\sqrt{n}\left(\frac{n}{e}\right)^n,
 \end{equation}
 and 
 \begin{equation}\label{rr4}
 1-\frac{\log 2\pi n}{2}\leq r_n\leq 1-\frac{\log 2\pi }{2}
 \end{equation}
  In each of the above inequalities equality hold if and only if $n=1.$
 \end{coro}
 \begin{proof} Replacing $x$ by $n\geq1$ and letting $q\longrightarrow1$ in (\ref{kh22}) we get (\ref{rr}), and after some standard computations we obtain inequalities (\ref{rr4}).
 \end{proof}
 \section{\textbf{Concluding Remarks}}  
 \noindent1. It is worth mentioning that the inequality (\ref{sss}) when letting $q$ tends to $1$, returns to the inequalities \cite{qi}
 \begin{equation}
 \frac{b^{b-1}}{a^{a-1}}e^{a-b}<\frac{\Gamma(a)}{\Gamma(b)}<\frac{b^{b-1/2}}{a^{a-1/2}}e^{a-b},
 \end{equation}
 for $b>a>0.$\\
 2. Let $n=2$ and $p_k=1/2,\;k=1,2$ in inequalities (\ref{222}) we obtain the lower bounds for the q-analogue for Gurland's ratio \cite{M7} as follows
 \begin{equation}\label{kh7}
 \begin{split} \frac{\left(\frac{1-q^{\frac{x+y}{2}}}{1-q}\right)^{x+y}}{\left(\frac{1-q^{x}}{1-q}\right)^{x}\left(\frac{1-q^{y}}{1-q}\right)^{y}}&\leq\frac{\Gamma_{q}^{2}\left(\frac{x+y+2}{2}\right)}{\Gamma_{q}(x+1)\Gamma_{q}(y+1)}\exp\left(\frac{\Li_{2}(1-q^{x})+\Li_{2}(1-q^{y})-2\Li_{2}(1-q^{x+y/2})}{\log q}\right)\\
 &\leq \frac{\left(\frac{1-q^{\frac{x+y}{2}}}{1-q}\right)^{x+y+1}}{\left(\frac{1-q^{x}}{1-q}\right)^{x+1/2}\left(\frac{1-q^{y}}{1-q}\right)^{y+1/2}}
 \end{split}
 \end{equation}
 where $x,y\in(0,\infty).$  In particular, let $q$ tends to $1$ in (\ref{kh7}) we get
 \begin{equation}\label{sana}
 \frac{\left(\frac{x+y}{2}\right)^{x+y}}{x^x y^y}\leq\frac{\Gamma^2\left(\frac{x+y+2}{2}\right)}{\Gamma(x+1)\Gamma(y+1)}\leq\frac{\left(\frac{x+y}{2}\right)^{x+y+1}}{x^{x+1/2} y^{y+1/2}}.
 \end{equation}
The left hand side inequalities of (\ref{sana}) has proved first by  Mortici \cite{MO} and the right hand side of inequalities (\ref{sana}) is new. 

\end{document}